\documentclass[12pt,draftcls,onecolumn]{IEEEtran}

\usepackage{amsmath}
\usepackage{amssymb}
\usepackage{amsfonts}
\usepackage{mathrsfs}
\usepackage{acronym}
\usepackage{epsfig}
\usepackage{cite}
\usepackage{graphicx}
\graphicspath{{Figures/}}
\usepackage{url}
\usepackage{stfloats}
\usepackage{amsopn}     
\usepackage{psfrag}
\usepackage[tight]{subfigure}
\usepackage{mdwtab}     
\usepackage{enumerate}
\usepackage{color}
\usepackage{cite}
\usepackage{balance}

\usepackage{acronym}
\usepackage{cite}

\newtheorem{theorem}{Theorem}
 \newtheorem{lemma}{Lemma}
\newtheorem{proposition}{Proposition}
 \newtheorem{definition}{Definition}
 \newtheorem{remark}{Remark}
 \newtheorem{assumption}{Assumption}

\newcommand{\BR}{\mathbb{R}}
\newcommand{\BP}{\mathbb{P}}
\newcommand{\BE}{\mathbb{E}}
\newcommand{\BZ}{\mathbb{Z}}
\newcommand{\BN}{\mathbb{N}}

\newcommand{\filt}{\mathscr{F}}

\newcommand{\sfmin}{\mathsf{min}}
\newcommand{\sfmax}{\mathsf{max}}

\let\oldmarginpar\marginpar
\renewcommand\marginpar[1]{\-\oldmarginpar[\raggedleft\footnotesize #1]%
{\raggedright\footnotesize #1}}

\title{Networked Decision Making for Poisson Processes: \\ Application to nuclear detection
\thanks{Manuscript submitted July 20, 2012.}}

\author{Chetan D. Pahlajani, 
		Ioannis Poulakakis 
		and Herbert G. Tanner
\thanks{Chetan D. Pahlajani is with the Department of Mathematical Sciences, University of Delaware, Newark, DE 19716, USA ({\tt chetan@math.udel.edu}).}
\thanks{Ioannis Poulakakis and Herbert G. Tanner are with the Department of Mechanical Engineering, University of Delaware, Newark, DE 19716, USA, ({\tt \{poulakas,btanner\}@udel.edu}).}}

\begin{document}

\acrodef{sprt}[\textsc{sprt}]{Sequential Probability Ratio Test}
\acrodef{sit}[\textsc{sit}]{Single-Interval Test}
\acrodef{snr}[\textsc{snr}]{Signal-to-Noise Ratio}
\acrodef{pi}[\textsc{pi}]{Program Investigator}
\acrodef{copi}[\textsc{c}o-\textsc{pi}]{Associate Program Investigator}
\acrodef{roc}[\textsc{roc}]{Receiver Operating Characteristic}
\acrodef{iid}[i.i.d.]{independent identically distributed}
\acrodef{lrt}[\textsc{lrt}]{likelihood ratio test}
\acrodef{cps}[cps]{counts per second}

\maketitle

\begin{abstract}
This paper addresses a detection problem where several spatially distributed sensors independently observe a time-inhomogeneous stochastic process. The task is to decide between two hypotheses regarding the statistics of the observed process at the end of a fixed time interval. In the proposed method, each of the sensors transmits once to a fusion center a locally processed summary of its information in the form of a likelihood ratio. The fusion center then combines these messages to arrive at an optimal decision in the Neyman-Pearson framework. The approach is motivated by applications arising in the detection of mobile radioactive sources, and offers a pathway toward the development of novel fixed-interval detection algorithms that combine decentralized processing with optimal centralized decision making.
\end{abstract}



\thispagestyle{plain}
\markboth{C. D. PAHLAJANI,  I. POULAKAKIS, H. G. TANNER}{POISSON DECISION IN SENSOR NETWORKS}

\section{Introduction}\label{S:Introduction}


Decision making is crucial in translating information to action. Human decision makers can be significantly assisted in determining a time-critical plan of action if provided with concise, dependable information. In a variety of applications---particularly those involving spatially and temporally nonuniform processes---the enhanced observational capabilities afforded by sensor networks render such solutions attractive for collecting the information required to arrive at a decision. Yet, the amount of data accumulated by a network of sensors is frequently overwhelming, underlining the need to filter and synthesize the collected data to ease decision making. Our goal in this work is to provide a framework that leverages local sensor-level information processing, to enable accurate fixed-interval decision making in the context of a spatially distributed sensor network which makes binary decisions based on observations of a time-inhomogeneous stochastic process.

Sensor networks---both mobile and static---have been employed in a wide range of applications, including environmental monitoring \cite{Sukhatme-2007-504, Leonard-gliders, Rus-cow}, intruder detection \cite{Howard-2006-446}, area coverage \cite{Cassandras-EJC, ZhongCassandras, CortezBullo-SICON, Stipanovic-cover, SB:NM:10}, source localization \cite{DandachBullo-09, Farrell-plume, optimotaxis, Morelande07}, and mapping of spatially distributed physical quantities \cite{info-surfing,fink-vijay}. Currently, networks of distributed sensors are used to trigger timely responses to a number of natural disasters, such as hurricanes \cite{hurricanes},  earthquakes \cite{oga99} and tsunamis \cite{dozier-tsunami, bernard-tsunami}. In the classical approach to network-based decision making, sensors relay the entirety of their observations to a central processing unit, which analyzes the data and issues a global decision. While this \textit{centralized} approach has the advantage of using all the information available, it does impose significant communication overhead. Alternatively, a \textit{decentralized} decision-making scheme \cite{BKP97, Tsi93, Var, VV97} can be used; in this setting, the sensors process measurement information and transmit a compressed version of it---typically in the form of a message with values in a finite alphabet---to a fusion center, which then provides a decision. Our motivation in this article stems from a class of problems associated with nuclear detection \cite{Ristic-SigProc,NDTW, Priedhorsky, lastline}; more specifically, the detection of illicit radioactive substances in transit. Remarkably, small or shielded quantities of nuclear material are very difficult to detect at a distance, due to the fact that their sensory signature is disguised in naturally occurring background radiation. Yet, the ability to provide fast and accurate decisions in such situations is of paramount importance to public safety and nuclear nonproliferation; see for instance \cite{Priedhorsky, lastline}, which make explicit reference to the need for networks of detectors deployed along transportation routes.

The physical quantities of interest in many applications (including nuclear detection \cite{Priedhorsky,NDTW,lastline}) can be captured by random processes characterized by discrete events that are highly localized in time. Phenomena of this sort can be mathematically modeled and analyzed within the framework of \textit{point processes} \cite{B81, DV1, DV2, SM91}. A realization of a point process is a random sequence of points, each representing the time and/or spatial location of an event. A point process can be characterized in terms of its \textit{intensity} which corresponds to the rate at which events occur. 
Beyond nuclear measurement, typical examples of such processes include customers to and from a service facility in queueing theory \cite{B81, CY01}, electron emission from a photodetector in optical communications systems \cite{GK95}, generation of electrical pulses in neurons \cite{Tuc}, and others. Of special interest in nuclear detection are Poisson processes which provide the natural models describing the emission and measurement of radiation \cite{NDTW,Ristic-SigProc,SVR}.

In regards to point processes, the problem of decision making between two alternative hypotheses (``all clear" versus ``alarm") has been addressed in \cite{BVW-II, B81, Rubin-IEEE-IT-72}; the solution typically involves the computation of a likelihood ratio, whose comparison against a threshold provides the decision. Error probability bounds for such decision problems are studied in \cite{HSS}, while robust decision making (in the presence of modeling uncertainty) is explored in \cite{GP85}. Decision problems with time-inhomogeneous point processes also arise in optical communications \cite{GK95,Ver86a}. In sum, for the classical (single observer) case, decision theory for general point processes is well-understood. However, the realization of these results in a \textit{network} setting requires care. Indeed, the likelihood ratios in \cite{B81,BVW-II} involve intensities computed on the basis of \textit{all} accumulated information,\footnote{Strictly speaking, the intensity is a \textit{conditional} rate at which events occur (conditioned on available information).} necessitating a modicum of caution in a setting such as ours where much of the computation and most of the raw data are decentralized (See Remark \ref{R:Contrast-Stoch-Int}).

The problem of detecting (moving and stationary) radioactive sources using networks of sensors has received a fair bit of attention in the literature. In situations where the parameters (location, trajectory, activity) of the source are unknown, Bayesian methods are frequently used \cite{BMT,Morelande07,NDTW,Ristic-SigProc}, embedding the issue of detection in a parameter estimation problem. While powerful, Bayesian methods for source parameter estimation exhibit computational complexity exponential in the number of parameters estimated, posing challenges for their implementation in real time for networks with more than ten nodes \cite{BMT,NDTW}. An important insight---and one that serves as the starting point for our analysis---is that in many cases of interest, the problem of source localization can be decoupled from the problem of source detection. Indeed, there are improved methods \cite{KMPS-ICRA2012,HLT-ICRA2012,WH-ICRA2012} for tracking the carrier of a potential radioactive source using sensor modalities other than a Geiger counter. Armed with this observation, source detection reduces to the problem of deciding whether the counts observed by a spatially distributed network of radiation sensors correspond solely to background radiation, or whether they also include emission from a radioactive source with \textit{known} parameters. In this setting, \cite{NDTW} explores the \ac{snr} resulting from the combination of data from a network of radiation sensors, 
allowing for spatially varying background rates. The analysis 
is restricted, however, to uniform linear source motion and does not provide a decision test. The costs and benefits of using networked sensors for moving sources, together with a threshold test (based on the total number of recorded counts) are addressed in \cite{stephens}, assuming uniform background and constant geometry between source and sensor.\footnote{Our analysis indicates that the optimal test involves comparing the likelihood ratio against a threshold, rather than the total number of counts.} For the case of a stationary source and correlated sensor measurements, a distributed detection scheme is developed in \cite{SVR} using the theory of copulas. The work in \cite{BMT} studies detection (via Bayesian estimation) for a moving source, but the motion is required to be linear with constant velocity. Detection and parameter estimation for an unknown number of static radioactive point sources are treated in \cite{Morelande07,Ristic-SigProc}. Evidently, the networked detection problem for general source motion with spatially varying background intensity has yet to be studied. 


Motivated by the above, we pose the following problem: a spatially distributed sensor network observes---over a fixed-time interval---a time-inhomogeneous point process which is known \emph{a priori} to be governed by one of two intensities. How should the local information be processed and communicated through the network in order to reach a reliable decision regarding which intensity governs the observed process? With respect to the spectrum of approaches from centralized to decentralized, we take in this paper an intermediate approach that combines the significantly lower communication cost of decentralized processing (not decision) with the enhanced accuracy of centralized decision making. For the case of a vector of Poisson processes whose intensities explicitly depend on time,\footnote{Time-dependence of intensities encodes the relative motion between the source and the sensors.} we develop an optimal---in the Neyman-Pearson sense---decision-making scheme that combines decentralized processing (local processing at each individual sensor) with centralized decision making via a fusion center. In particular, assuming that the relative motion between the suspected source and the sensors is deterministic and known, and that the sensor observations are conditionally independent, our method relies on the sensors communicating processed information in the form of locally-computed likelihood ratios to the fusion center. The fusion center then combines these messages to arrive at a decision, without the need for any additional information such as the location or the raw data of individual sensors. 
As applied to radiation detection, our framework allows us to consider  arbitrary continuous source motion in any number of dimensions allowing for sensor mobility and spatially varying background rates. In relation to sensor networks, the time-inhomogeneity in our problem leads to non-identically distributed sensor observations; this marks a departure from the frequently used \ac{iid} assumption.

The paper is organized as follows. In Section \ref{S:Problem-statement}, we state the problem and our technical assumptions. The main result (Theorem~\ref{T:New-Main-Result}), which indicates how a global likelihood ratio test can be formulated based on local computation of sensor-specific likelihood ratios, is presented in Section \ref{S:Main-Result}. The proof of this result and the supporting technical material are found in Section \ref{S:Proofs}. Based on this analysis, we offer conservative lower and upper bounds on the probabilities of detection and false alarm, respectively, in Section \ref{section:bounds}. Finally, a numerical example of a one-dimensional case of networked nuclear detection is developed in Section \ref{S:Example}, highlighting the benefits of using multiple sensors. The results provided in this paper can be viewed as a building block toward a general decision-making framework that leverages networks of mobile sensor platforms to enhance detection capability in problems that involve time-inhomogeneous point processes.


\section{Problem Statement and Assumptions}\label{S:Problem-statement}

Consider a collection of sensors observing a point process generated by some physical phenomenon. The goal is to decide between two hypotheses regarding the state of the environment. To this end, each sensor communicates a processed version of its observations to a fusion center, which combines all received messages to a binary decision (Fig. \ref{fig:architecture}). With an eye towards applications such as detection of mobile radioactive sources where the point process can be observed only for a limited time, we require here that the decision be made within a fixed time interval. The problem is formulated as a binary hypothesis test based on measurements from an array of $k$ sensors connected in a parallel network architecture. Figure~\ref{fig:architecture} shows a realization of such a network for radiation detection.

\begin{figure}[h!]
\center
\includegraphics[width=0.45\columnwidth]{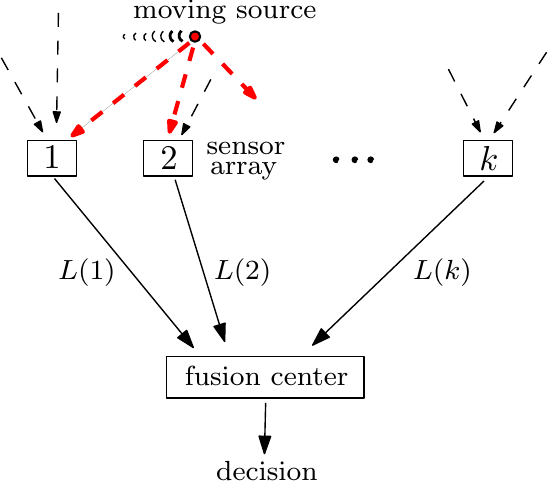}
 \caption{The architecture of the detection scheme as applied to radiation detection.  There is a network of sensors each counting the total number of rays that have  arrived at them.  Thick arrows represent rays emitted from a moving radioactive source and thin arrows mark rays from background radiation. Based on the number and timing of arrival of those counts, each sensor computes a likelihood ratio $L(i)$ which is then transmitted to a fusion center.  The fusion center combines this information to make a decision regarding the presence of the radioactivity in the target moving in front of the sensors.\label{fig:architecture}}
\end{figure}

Let $H_0$ and $H_1$ denote the two hypotheses regarding the state of the environment; for example, the absence or presence of a radioactive source on the moving target. 

\begin{assumption}
The reading at sensor $i$, where $1 \le i \le k$, is a non-decreasing $\BZ^+$-valued piecewise constant, right-continuous function, which increases in steps of size one.  
\end{assumption}

\begin{assumption}
Conditioned on hypothesis $H_j$, $j \in \{0,1\}$, the observations at distinct sensors are independent.  
\end{assumption}

Sensor observations are modeled as inhomogeneous Poisson processes, whose intensities are deterministic functions of time. In particular, for $1 \le i \le k$, sensor $i$ observes a Poisson process, whose time-dependent intensity under $H_0$ is $\beta_i(\cdot)$, while under $H_1$ is $\beta_i(\cdot)+\nu_i(\cdot)$. The explicit dependence of the intensities on time arises from the known motion of the target, which is a potential ``source" of the point process with intensity $\nu_i$. 
Thus sensor observations are independent but not identically distributed.  This is in agreement
with the physics of the motivating application, since a gamma ray emitted from the source
cannot pass through more than one sensor simultaneously (see Fig.~\ref{fig:architecture}),
and the time varying nature of the distance between source and sensor changes the arrival
statistics on the sensor side \cite{NDTW}.

The following assumptions will be imposed on $\beta_i$ and $\nu_i$.

\begin{assumption}\label{A:Process-under-H0}
For $1 \le i \le k$, $\beta_i:[0,T] \to [\beta_\sfmin,\beta_\sfmax]$ is a bounded, continuous function with $0 < \beta_\sfmin < \beta_\sfmax < \infty$, $\beta_\sfmin$, $\beta_\sfmax$ independent of $i \in \{1,2,\dots,k\}$.
\end{assumption}

\begin{assumption}\label{A:Process-under-H1}
For $1 \le i \le k$, $\nu_i:[0,T] \to [\nu_\sfmin,\nu_\sfmax]$ is a bounded, continuous function with $0 < \nu_\sfmin < \nu_\sfmax < \infty$, $\nu_\sfmin$, $\nu_\sfmax$ independent of $i \in \{1,2,\dots,k\}$.
\end{assumption}

The physical significance of these assumptions will become clear in Section \ref{S:Example}.

The detection problem can now be summarized as follows: Suppose that $T > 0$ is the \emph{decision time}; that is, the time by which a decision must be made. Given a single realization of a $k$-dimensional vector of Poisson processes over the time horizon $[0,T]$ (the $k$ components corresponding to the $k$ sensors), decide whether the intensities are given by the collection $\beta_i(\cdot)$ or by the collection $\beta_i(\cdot)+\nu_i(\cdot)$, $1 \le i \le k$.


\section{Main Result}\label{S:Main-Result}


In this section, we state our main result, Theorem \ref{T:New-Main-Result}. For a collection of sensors with a fusion center, configured in a parallel network architecture, Theorem \ref{T:New-Main-Result} gives a procedure for locally processing sensor information 
and transmitting compressed summaries (at a single time) to the fusion center to enable networked decision making that recovers the performance of a centralized scheme. Our starting point is a measurable space $(\Omega,\filt)$,\footnote{Here, $\Omega$ is the sample space and $\filt$ is a $\sigma$-field on $\Omega$. Equipping $(\Omega,\filt)$ with a probability measure $\BP$ gives the probability space $(\Omega,\filt,\BP)$.} on which a $k$-dimensional vector of counting processes $N_t = (N_t(1),\dots,N_t(k))$, $t \in [0,T]$ is defined. In our problem, $N_t(i)$ is the number of counts registered at sensor $i \in \{1,2,\dots,k\}$ up to (and including) time $t \in [0,T]$. The two hypotheses $H_0$ and $H_1$ 
regarding the state of the environment correspond to two distinct probability measures on $(\Omega,\filt)$. Hypothesis $H_0$ corresponds to a probability measure $\BP_0$, with respect to which the $N_t(i)$, $1 \le i \le k$, are independent Poisson processes
over $t \in [0,T]$ with intensities $\beta_i(t)$, respectively. Hypothesis $H_1$ corresponds to a probability measure $\BP_1$, with respect to which the $N_t(i)$, $1 \le i \le k$, are independent Poisson processes
over $t \in [0,T]$ with intensities $\beta_i(t) + \nu_i(t)$, respectively. The decision problem is thus one of identifying the correct probability measure ($\BP_0$ versus $\BP_1$) on $(\Omega,\filt)$ based on a realization of the $k$-dimensional process $N_t = (N_t(1),\dots,N_t(k))$.

We will keep track of the flow of information using the filtration $(\filt^N_t:0 \le t \le T)$ generated by the process $N_t$; here, for $t \in [0,T]$, $\filt^N_t = \sigma(N_s:0 \le s \le t)$ is the smallest $\sigma$-field on $\Omega$ with respect to which all the ($k$-dimensional) random variables $N_s$, $0 \le s \le t$ are measurable. The interpretation is: for any event $A \in \filt^N_t$, an observer of the sample path $s \mapsto N_s$, $0 \le s \le t$, knows at time $t$ whether or not the event $A$ has occurred. The $\sigma$-field $\filt^N_T$ thus represents the information generated by the totality of sensor observations up to $t=T$; to wit, the information on which the decision must be based.

A test for deciding between hypotheses $H_0$ and $H_1$ on the basis of $\filt^N_T$ observations can be thought of as a set $A_1 \in \filt^N_T$ with the following significance: if the outcome $\omega \in A_1$, decide $H_1$; if $\omega \in A_0 \triangleq \Omega\setminus A_1$, decide $H_0$. For a test $A_1 \in \filt^N_T$, two types of errors might occur. A ``false alarm" occurs when the outcome $\omega \in A_1$ (i.e. decide $H_1$) while $H_0$ is the correct hypothesis. A ``miss" occurs when $\omega \in \Omega\setminus A_1$ (i.e. decide $H_0$) while $H_1$ is the correct hypothesis. Clearly, the probability of false alarm is given by $\BP_0(A_1)$, while the probability of a miss is given by $\BP_1(\Omega\setminus A_1)$. Then, the probability of detection is given by $\BP_1(A_1)=1-\BP_1(\Omega\setminus A_1)$. 

In the Neyman-Pearson framework, one is given an acceptable upper bound on the probability of false alarm $\alpha \in (0,1)$, and the problem is to find an optimal test:\footnote{We restrict attention here to tests without randomization; see \cite{Leh, P94}.} a set $A_1^* \in \filt^N_T$ which maximizes the probability of detection over all tests whose probability of false alarm is less than or equal to $\alpha$. The following result provides an optimal test that employs local information processing at the sensor level, to enable decisions at the fusion center that recover the optimal performance of a centralized Neyman-Pearson test. 

\begin{theorem}[Main Result]\label{T:New-Main-Result}
Consider a network with $k$ sensors and a fusion center connected in the parallel configuration of Fig. \ref{fig:architecture}. For $1 \le i \le k$, let $N_t(i)$, $t \in [0,T]$ denote the observation at sensor $i$ over the time interval $[0,T]$ and let $(\tau_n(i):n \ge 1)$ be the jump times of $N_t(i)$. Assume that at decision time $T$, sensor $i$ transmits to the fusion center the statistic
\[
L_T(i) \triangleq \exp\left(-\int_0^T \nu_i(s)ds\right) \prod_{n=1}^{N_T(i)}\left(1+\frac{\nu_i(\tau_n(i))}{\beta_i(\tau_n(i))}\right)
\]
computed on the basis of its observation $t \mapsto N_t(i)$, $t \in [0,T]$.
Then, the test $A_1^* = \{L_T \ge \gamma\}$ performed at the fusion center, with
\[
L_T \triangleq \prod_{i=1}^k L_T(i)
\] 
and $\gamma>0$ satisfying $\BP_0(L_T \ge \gamma)=\alpha$,\footnote{While there may not exist such $\gamma$ for every $\alpha \in (0,1)$, one can always find a sequence $\alpha_n \to 0$ for which there exist $\gamma_n>0$ with $\BP_0(L_T \ge \gamma_n)=\alpha_n$. In other words, arbitrarily small upper bounds on probability of false alarm can be accommodated.} 
is optimal for $\filt^N_T$-observations in the sense that for any $A_1 \in \filt^N_T$ with $\BP_0(A_1) \le \alpha$, we have $\BP_1(A_1^*) \ge \BP_1(A_1)$.
\end{theorem}

Before continuing with the proof of Theorem \ref{T:New-Main-Result}, we highlight
why optimal decision making through decentralized processing
of information at the sensor level is possible in the case considered here.


\begin{remark}\label{R:Main-Result-Significance}
The decision test $\{L_T \ge \gamma\}$ is optimal (in the Neyman-Pearson sense) for $\filt^N_T$-observations, the latter comprising the totality of information in the $k$ waveforms $t \mapsto N_t(i)$, $t \in [0,T]$. Equivalently, if one were to consider a centralized framework (where information is continuously streamed from each of the sensors to the fusion center), $\{L_T \ge \gamma\}$ would be the optimal test. Note, however, that each $L_T(i)$ in the product $L_T = \prod_{i=1}^k L_T(i)$ can be computed locally at sensor $i$ without any knowledge of the measurements at other sensors (see also Remark \ref{R:Contrast-Stoch-Int}). Indeed, computation of $L_T(i)$  requires knowledge solely of the times $(\tau_n(i):n \ge 1)$ at which counts have been recorded at sensor $i$, and the quantities $\beta_i(t)$, $\nu_i(t)$, which are deterministic. Consequently, each sensor simply needs to transmit its locally computed $L_T(i)$ to the fusion center at (the single) time $t=T$ (in lieu of the element $t \mapsto N_t(i)$, $t \in [0,T]$ of function space). The fusion center then forms the product $L_T = \prod_{i=1}^k L_T(i)$ which is compared against $\gamma$ to arrive at a decision. We thus retain the accuracy of centralized decision making while decentralizing most of the data processing, thereby accruing significant savings in communication costs.
\end{remark}


\section{Proof of Main Result}\label{S:Proofs}

The contents of this section are organized as follows. In Section \ref{SS:Definitions}, we provide precise definitions for various quantities of interest. In Section \ref{SS:Lemmas}, we state and prove some results needed for the proof of Theorem \ref{T:New-Main-Result}. In Section \ref{SS:Proof-Main}, the proof of Theorem \ref{T:New-Main-Result} is completed.

\subsection{Definitions}\label{SS:Definitions}
Let $(\Omega,\filt,\BP)$ be a probability space. The \textit{sample space} $\Omega$ is the set of all possible outcomes $\omega$ of a random experiment, $\filt$ is the $\sigma$-field (or $\sigma$-algebra) of \textit{events}, and $\BP$ is a probability measure. A \textit{point process} on $[0,\infty)$ can be described by a sequence $(\tau_n:n \ge 0)$ of random variables defined on $(\Omega,\filt,\BP)$ taking values in $[0,\infty]$ such that
\begin{eqnarray}
\tau_0  & \equiv  & 0\enspace,\\
\tau_n  & < & \infty \hspace{.2in} \Rightarrow \hspace{.2in} \tau_n < \tau_{n+1}\enspace.
\end{eqnarray}
Here, $\tau_n$ denotes the (random) time of the $n$-th occurrence of an event (such as a radiation counter registering a count). Associated to the sequence $(\tau_n:n \ge 0)$ is the stochastic process $(N_t:t \ge 0)$ defined by
\begin{equation}
N_t  \triangleq \sum_{n \ge 1} 1_{(\tau_n \le t)}\enspace,
\end{equation}
where $1_A$ denotes the indicator function of $A$, i.e.
\[
1_A(\omega) \triangleq \begin{cases}1 \quad &\text{if $\omega \in A$,}\\0 \quad &\text{else.}\end{cases}
\]
Thus, $N_t$ counts the number of occurrences of the phenomenon prior to or at time $t$. $N_t$ is a \textit{counting process}, i.e. $N_t$ is a $\BZ^+$-valued process with $N_0=0$ such that the sample paths $t \mapsto N_t$ are non-decreasing, piecewise constant, right-continuous functions of $t$ which increase in steps of size 1. We will also refer to $N_t$ as a point process \cite{B81}. 

A \textit{filtration} $(\filt_t:t \ge 0)$ is an increasing family of sub-$\sigma$-fields of $\filt$, i.e. $\filt_t \subset \filt$ for all $t$, and $s \le t$ implies $\filt_s \subset \filt_t$. The $\sigma$-field $\filt_t$ represents the information available at time $t$. A stochastic process $(N_t:t \ge 0)$ taking values in $\BR^k$ is said to be \textit{adapted} to the filtration $\filt_t$ if for all $t \ge 0$, $N_t$ is $\filt_t$-measurable; i.e. for any Borel measurable subset $B$ of $\BR^k$, the event $\{N_t \in B\} \in \filt_t$. For $t \ge 0$, let $\filt^N_t = \sigma(N_s:0 \le s \le t)$ be the smallest $\sigma$-field on $\Omega$ with respect to which all the random variables $N_s$, $0 \le s \le t$ are measurable. Then, $(\filt^N_t:t \ge 0)$ is the filtration generated by the process $N_t$ and corresponds to the information available to an observer of the process $N_t$. Clearly, if $N_t$ is $\filt_t$-adapted, then $\filt^N_t \subset \filt_t$ for all $t$. Next, let us make precise what we mean by an inhomogeneous Poisson process.

\begin{definition}[Inhomogeneous Poisson process]\label{D:Poisson-process}
Suppose $\lambda(t)$ is a nonnegative, measurable function such that $\int_0^t \lambda(s)ds < \infty$ for all $t>0$. A point process $N_t$ on a probability space $(\Omega,\filt,\BP)$ adapted to the filtration $(\filt_t:t \ge 0)$ is said to be a $(\BP,\filt_t)$-Poisson process with \textit{intensity} $\lambda(t)$ if for $0 \le s \le t$,
\begin{enumerate}
\item $N_t - N_s$ is independent of $\filt_s$, and
\item $N_t - N_s$ is a Poisson random variable with parameter $\int_s^t \lambda(u)du$, i.e. for all $n \in \BZ^+$,
\begin{equation}
\BP(N_t - N_s = n) = e^{-\int_s^t \lambda(\tau)d\tau}\frac{\left(\int_s^t \lambda(\tau)d\tau\right)^n}{n!}\enspace .
\end{equation}
\end{enumerate}
\end{definition}

\subsection{Useful Results}\label{SS:Lemmas}

The primary result in this section is Proposition \ref{P:Math-setting}, which provides the probabilistic setup for the statement and proof (given in the next section) of Theorem \ref{T:New-Main-Result}. Our development proceeds through the following steps. We start with Proposition \ref{P:VI.2.T3}---a result of Br\'emaud \cite{B81} which plays a fundamental role in our analysis. To apply Proposition \ref{P:VI.2.T3}, we first relate our problem---posed on the time interval $[0,T]$---to a corresponding problem over the time interval $[0,1]$, as in Proposition \ref{P:VI.2.T3}. This is accomplished by  a time-rescaling argument (see proof of Proposition \ref{P:Math-setting}). Lemma \ref{L:Poisson-rescaling-time} aids us in this regard by describing how the intensity of a Poisson process transforms under a change of time. Next, we verify in Lemma \ref{L:L_t-martingale} that condition (\ref{E:LRatio-mean-one}) in the statement of Proposition \ref{P:VI.2.T3} holds. It is important to note that Proposition \ref{P:VI.2.T3} and Lemmas \ref{L:Poisson-rescaling-time} and \ref{L:NP-Lemma-B81} are formulated in terms of a general probability space, not necessarily identical to the one in Proposition \ref{P:Math-setting} (which supports our processes of interest). Lemma \ref{L:L_t-martingale}, on the other hand, pertains to the \textit{specific} setup of Proposition \ref{P:Math-setting}. Finally, Lemma \ref{L:NP-Lemma-B81} recalls the Neyman-Pearson Lemma, which is used in the next section in the proof of Theorem \ref{T:New-Main-Result}. 

In allowing stochastic intensities, as in Proposition \ref{P:VI.2.T3} below, one has to have the technical requirement of \textit{predictability} \cite[Section I.3]{B81}. In our problem, however, the intensities are deterministic and automatically predictable. In terms of notation, we follow the convention of using $\lambda_t$ for a stochastic intensity versus $\lambda(t)$ for a deterministic one.

\begin{proposition}[Theorem VI.2.T3, \cite{B81}]\label{P:VI.2.T3}
Let $(N_t(1),\dots,N_t(k))$ be a $k$-variate point process adapted to the filtration $\filt_t$ on the given probability space $(\Omega,\filt,\BP_0)$, and let $\lambda_t(i)$, $1 \le i \le k$, be the predictable $(\BP_0,\filt_t)$-intensities of $N_t(i)$, $1 \le i \le k$, respectively. Let $\mu_t(i)$, $1 \le i \le k$, be nonnegative, $\filt_t$-predictable processes such that for all $t \ge 0$, $1 \le i \le k$,
\begin{equation}
\int_0^t \mu_s(i) \lambda_s(i)ds < \infty \qquad \BP_0-a.s.
\end{equation}
Let $(\tau_n(i):n \ge 1)$ denote the jump times of $N_t(i)$, $1 \le i \le k$, and define the process $L_t$ by
\begin{equation}
L_t = \prod_{i=1}^k L_t(i)\enspace,
\end{equation}
where
\begin{equation}
L_t(i) = \left(\prod_{n=1}^{N_t(i)} \mu_{\tau_n(i)}(i)\right) \exp\left\{\int_0^t (1-\mu_s(i))\lambda_s(i)ds\right\}\enspace,
\end{equation}
with the convention that $\prod_{n=1}^0 (\dots)=1$. Suppose moreover that
\begin{equation}\label{E:LRatio-mean-one}
\BE_0[L_1]=1 \enspace .
\end{equation}
Define the probability measure $\BP_1$ on $(\Omega,\filt)$ by
\begin{equation}
\frac{d\BP_1}{d\BP_0} = L_1 \enspace .
\end{equation}
Then, for each $1 \le i \le k$, $N_t(i)$ has the $(\BP_1,\filt_t)$-intensity $\tilde{\lambda}_t(i)=\mu_t(i)\lambda_t(i)$ over $[0,1]$.
\end{proposition}

To underscore the role that deterministic intensities play in allowing decentralized processing in our problem, we briefly discuss some of the subtleties that arise in networked detection of point processes with stochastic intensities.

\begin{remark}\label{R:Contrast-Stoch-Int}
In Proposition \ref{P:VI.2.T3}, the assumption of $\filt_t$-predictability of $\lambda_t(i)$, $\mu_t(i)$ (which are in general stochastic) implies that the latter potentially depend on all the information in $\filt_t$, which includes the information generated by \textit{all} the sample paths $s \mapsto N_s(i)$, $s \in [0,t]$, $1 \le i \le k$.\footnote{$\filt_t$ may contain information generated by other random quantities too.} Consequently, for a sensor network which observes a point process with stochastic intensities, computation of the $L_t(i)$'s cannot be decentralized as described in Remark \ref{R:Main-Result-Significance} without incorporating a filtering component \cite[Section VI.4]{B81}---finding the best estimates $\hat{L}_t(i)$ of $L_t(i)$ based on the locally available information $\filt^{N(i)}_t \triangleq \sigma(N_s(i):0 \le s \le t)$. In such a decentralized processing scheme, the likelihood ratio to be compared to a threshold (at the fusion center) would be the process $\hat{L}_t \triangleq \prod_{i=1}^k \hat{L}_t(i)$ evaluated at the decision time. This would entail, in general, some loss of performance in comparison to a fully centralized scheme. These extra considerations do not arise in our detection problem since our analysis focuses on a problem with deterministic intensities.
\end{remark}

Next, we state Lemmas \ref{L:L_t-martingale} and \ref{L:Poisson-rescaling-time}, which together verify \eqref{E:LRatio-mean-one} for our problem. The proofs of these Lemmas are given in the Appendix. Lemma \ref{L:L_t-martingale} shows that the process $(L_t:t \in [0,T])$ defined by \eqref{E:LRatio-specific}--\eqref{E:LRatio-i-specific} is a martingale with mean one. Lemma \ref{L:Poisson-rescaling-time} describes how the intensity of a Poisson process transforms under a time rescaling. Taken together, these lemmas enable us in Proposition \ref{P:Math-setting} to obtain the probability measure $\BP_1$ on $(\Omega,\filt)$ (corresponding to hypothesis $H_1$). In the sequel, we denote the expectations corresponding to probability measures $\BP_0$ and $\BP_1$ by $\BE_0$ and $\BE_1$, respectively. 

\begin{lemma}\label{L:L_t-martingale}
Under Assumptions \ref{A:Process-under-H0}, \ref{A:Process-under-H1}, the process $(L_t:t \in [0,T])$ defined by
\begin{equation}\label{E:LRatio-specific}
L_t \triangleq \prod_{i=1}^k L_t(i)
\end{equation}
where
\begin{equation}\label{E:LRatio-i-specific}
L_t(i) \triangleq \exp\left(-\int_0^t \nu_i(s)ds\right) \prod_{n=1}^{N_t(i)}\left(1+\frac{\nu_i(\tau_n(i))}{\beta_i(\tau_n(i))}\right) \enspace ,
\end{equation}
is a nonnegative $(\BP_0,\filt^N_t)$-martingale. By $(\BP_0,\filt^N_t)$-martingale, we mean that  
\begin{enumerate}
\item $L_t$ is adapted to the filtration $\filt^N_t$,
\item $\BE_0 [|L_t|]<\infty$ for all $t \in [0,T]$, 
\item For $0 \le s \le t \le T$, $\BE_0[L_t|\filt^N_s]=L_s$, $\BP_0$-a.s. 
\end{enumerate}
Thus, $L_t$ has constant mean, i.e. $\BE_0[L_t]=\BE_0[L_0]=1$ for all $t \in [0,T]$.
\end{lemma}
\begin{proof} See Appendix.
\end{proof}

Suppose $(\Omega,\filt,\BP)$ is a probability space equipped with the filtration $(\filt_t:t \ge 0)$. Let $\lambda(t)$ be a nonnegative, measurable function defined on $[0,\infty)$ with $\int_0^t \lambda(s)ds < \infty$ for all $t>0$. Let $(X_t:t \ge 0)$ be a $(\BP,\filt_t)$-Poisson process with intensity $\lambda(t)$ (see Definition \ref{D:Poisson-process}). 

\begin{lemma}\label{L:Poisson-rescaling-time}
Fix $T>0$. Let $u = t/T$. Let $Y_u = X_t = X_{T\cdot u}$, $\mathscr{G}_u = \filt_t = \filt_{T\cdot u}$ for $u \ge 0$. Define $\tilde{\lambda}(u)$ on $[0,\infty)$ by
\begin{equation}
\tilde{\lambda}(u) = T \lambda(T\cdot u) 
\end{equation}
for $u \ge 0$. Then, $(Y_u:u \ge 0)$ is a $(\BP,\mathscr{G}_u)$-Poisson process with intensity $\tilde{\lambda}(u)$.
\end{lemma}
\begin{proof} See Appendix.
\end{proof}

We now state Proposition \ref{P:Math-setting}, which constructs a probability measure $\BP_1$ on $(\Omega,\filt)$ corresponding to hypothesis $H_1$. This completes the construction of the probabilistic model of our networked decision problem. As will be seen in the next section, the explicit construction of $\BP_1$ via the process $L_t$ and the probability measure $\BP_0$, as described in the proof of Proposition \ref{P:Math-setting}, facilitates the application of Lemma \ref{L:NP-Lemma-B81} in proving Theorem \ref{T:New-Main-Result}.   

\begin{proposition}\label{P:Math-setting}
Suppose $(\Omega,\filt,\BP_0)$ is a probability space, on which
 $N_t=(N_t(1),\dots,N_t(k))$, $t \in [0,T]$, is a vector of independent $\filt^N_t$-Poisson processes whose components $N_t(i)$ admit intensities $\beta_i(t)$, $1 \le i \le k$, with $\beta_i(t)$ satisfying Assumption \ref{A:Process-under-H0}. Then, there exists a probability measure $\BP_1$ on $(\Omega,\filt)$ with $\BP_1 \ll \BP_0$, with respect to which $N_t(i)$, $1 \le i \le k$, are independent $\filt^N_t$-Poisson processes over $t \in [0,T]$ with intensities $\beta_i(t)+\nu_i(t)$, with $\nu_i(t)$ satisfying Assumption \ref{A:Process-under-H1}.
\end{proposition}

\begin{proof}
Recall the process $(L_t:t \in [0,T])$ defined by \eqref{E:LRatio-specific}, \eqref{E:LRatio-i-specific}. Lemma \ref{L:L_t-martingale} assures us that $L_T$ is a nonnegative random variable with $\BE_0[L_T]=1$. Hence, $\BP_1$ defined through
\[
\BP_1(A) \triangleq \int_A L_T(\omega) \BP_0(d\omega) \quad \text{for $A \in \filt$,}
\]  
is indeed a probability measure on $(\Omega,\filt)$ which is absolutely continuous with respect to $\BP_0$. We would now like to show that with respect to $\BP_1$, $N_t(i)$ for $1 \le i \le k$, are independent Poisson processes over $t \in [0,T]$ with intensities $\beta_i(t)+\nu_i(t)$. To enable the application of Proposition \ref{P:VI.2.T3}, we use a time rescaling argument. 

Let $u \triangleq t/T$ be a rescaled time variable taking values in $[0,1]$. For $1 \le i \le k$, $u \in [0,1]$, let $\tilde{N}_u(i) = N_t(i) = N_{T\cdot u}(i)$. Let $\tilde{N}_u = (\tilde{N}_u(1),\dots,\tilde{N}_u(k))$ and let $\mathscr{G}_u \triangleq \sigma(\tilde{N}_v:0 \le v \le u)$ for $u \in [0,1]$. By Lemma \ref{L:Poisson-rescaling-time}, each  $\tilde{N}_u(i)$, $1 \le i \le k$ is a $(\BP_0,\mathscr{G}_u)$-Poisson process with intensity $\tilde{\beta}_i(u) \triangleq T\beta_i(T \cdot u)$. The independence of $\tilde{N}_u(i)-\tilde{N}_v(i)$ and $\tilde{N}_p(j)-\tilde{N}_q(j)$, $i \neq j$ follows from the independence of $N_{T\cdot u}(i)-N_{T\cdot v}(i)$ and $N_{T\cdot p}(j)-N_{T\cdot q}(j)$. Denote by $(\tilde{\tau}_n(i):n \ge 1)$ the sequence of jump times of $\tilde{N}_u(i)$. Note that $\tilde{\tau}_n(i) = \tau_n(i)/T$. Letting $\tilde{\nu}_i(u) \triangleq T\nu_i(T\cdot u)$ for $1 \le i \le k$, define a process $(\tilde{L}_u:u \in [0,1])$ by  
\begin{equation}
\tilde{L}_u \triangleq \prod_{i=1}^k \tilde{L}_u(i)
\end{equation}  
with
\begin{equation}
\tilde{L}_u(i) \triangleq \exp\left(-\int_0^u \tilde{\nu}_i(s)ds\right) \prod_{n=1}^{\tilde{N}_u(i)}\left(1+\frac{\tilde{\nu}_i(\tilde{\tau}_n(i))}{\tilde{\beta}_i(\tilde{\tau}_n(i))}\right) \enspace .
\end{equation}
It is now easily checked that $L_t = \tilde{L}_{t/T}$ for all $t \in [0,T]$, which implies in particular that 
\begin{equation}
\frac{d\BP_1}{d\BP_0} = \tilde{L}_1 \enspace .
\end{equation}
We now apply Proposition \ref{P:VI.2.T3} \cite[Theorem VI.2.T3]{B81}, using the \textit{rescaled} time variable $u \in [0,1]$, filtration $\mathscr{G}_u$, with  $\tilde{\beta}_i(u)$ and $ 1+ \tilde{\nu}_i(u)/\tilde{\beta}_i(u)$ in place of $\lambda_t(i)$ and  $\mu_t(i)$ respectively, to infer that $\tilde{N}_u(i)$ has intensity $\tilde{\beta}_i(u)+\tilde{\nu}_i(u)$ with respect to $\BP_1$. Using Lemma \ref{L:Poisson-rescaling-time} in the ``reverse" direction (i.e. interchanging $t$ and $u$, replacing $T$ by $1/T$), it now follows that the $N_t(i)$'s, $1 \le i \le k$, have $(\BP_1,\filt^N_t)$-intensities $\beta_i(t)+\nu_i(t)$, respectively. To complete the proof, it remains to show that the $N_t(i)$, $1 \le i \le k$, are independent and Poisson under $\BP_1$. Using Assertion ($\beta$) of \cite[Theorem II.3.T8]{B81} (with $X_t \equiv 1$), we get that for $1 \le i \le k$, 
\begin{equation}
N_t(i) - \int_0^t \left[\beta_i(s) + \nu_i(s)\right]ds  
\end{equation}
is a $(\BP_1,\filt^N_t)$-martingale. By the Multichannel Watanabe Theorem \cite[Theorem II.2.T6]{B81}, it now follows that with respect to $\BP_1$, the $N_t(i)$'s are independent $\filt^N_t$-Poisson processes over $t \in [0,T]$  with intensities $\beta_i(t)+\nu_i(t)$, respectively. 
\end{proof}

Before concluding this section, we state the Neyman-Pearson Lemma which will be used in the next section to prove Theorem \ref{T:New-Main-Result}. The Neyman-Pearson Lemma describes an optimal rule for deciding between probability measures $\BP_0$ and $\BP_1$ on a measurable space $(\Omega,\filt)$ on the basis of observations in the sub-$\sigma$-field $\mathscr{G} \subset \filt$. Thus, for any event $A \in \mathscr{G}$, it is known whether or not $A$ has occurred. Recall that $\BP_1$ is \textit{absolutely continuous} with respect to $\BP_0$, both restricted to $(\Omega,\mathscr{G})$, denoted $\BP_1 \ll \BP_0$ if, whenever $A \in \mathscr{G}$ with $\BP_0(A)=0$, we have $\BP_1(A)=0$. 

\begin{lemma}[Neyman-Pearson Lemma, Theorem VI.1.T1, \cite{B81}]\label{L:NP-Lemma-B81}
For $\alpha \in (0,1)$, suppose $\gamma$ is a real number such that
\begin{equation}
\BP_0(L \ge \gamma)=\alpha
\end{equation}
where $L$ is the Radon-Nikodym derivative of $\BP_1$ with respect to $\BP_0$, both probabilities on $(\Omega,\mathscr{G})$. Then the decision strategy $A_1^* = \{L \ge \gamma\}$ is optimal for $\mathscr{G}$-observations in the sense that for any $A_1 \in \mathscr{G}$ with $\BP_0(A_1) \le \alpha$, we have $\BP_1(A_1^*) \ge \BP_1(A_1)$.
\end{lemma}

See \cite{B81} for a proof. We will use this Lemma in the next section to prove Theorem \ref{T:New-Main-Result}, with $\filt^N_T$ playing the role of $\mathscr{G}$.

\subsection{Proof of Theorem \ref{T:New-Main-Result}}\label{SS:Proof-Main}

Let $\hat{\BP}_0$, $\hat{\BP}_1$ be the restrictions of $\BP_0$, $\BP_1$ respectively to $\filt^N_T$. Since $\BP_1 \ll \BP_0$ on $\filt$, the restrictions of $\BP_1$ and $\BP_0$ to the smaller $\sigma$-field $\filt^N_T$ inherit the absolute continuity, i.e. $\hat{\BP}_1\ll \hat{\BP}_0$. Hence, the Radon-Nikodym derivative exists; i.e.\ there exists a nonnegative, $\filt^N_T$-measurable random variable $\xi$, denoted $\frac{d\hat{\BP}_1}{d\hat{\BP}_0}$, such that for any $A \in \filt^N_T$,
\begin{equation}\label{E:RNDeriv}
\BP_1(A) = \int_A \xi(\omega) \BP_0(d\omega)\enspace.
\end{equation}
Moreover, this Radon-Nikodym derivative is unique in the sense that if $\tilde{\xi}$ is any nonnegative, $\filt^N_T$-measurable random variable satisfying (\ref{E:RNDeriv}) with $\tilde{\xi}$ replacing $\xi$, then $\xi=\tilde{\xi}$, $\BP_0$-a.s. Since $L_T$ is a nonnegative, $\filt^N_T$-measurable random variable (by Lemma \ref{L:L_t-martingale}) which satisfies (\ref{E:RNDeriv}), it follows that
\begin{equation}
L_T = \frac{d\hat{\BP}_1}{d\hat{\BP}_0}\enspace,
\end{equation}
$\BP_0$-a.s. A direct application of Lemma \ref{L:NP-Lemma-B81} completes the proof.

\section{Performance Analysis}\label{section:bounds}

Here we provide a lower bound on the probability of detection and an upper bound on the probability of false alarm when the proposed detection scheme is used. It turns out that bounds on both these probabilities involve the tails of (different) Poisson distributions. To compactly describe our results, we follow the notation of \cite{G87}.

\begin{definition}[Poisson Tails]
For $\lambda >0$, $j \in \BZ^+$, let $p(\lambda,j)$ denote the Poisson distribution
\[
p(\lambda,j) = e^{-\lambda} \frac{\lambda^j}{j!}~.
\]
The left and right tail probabilities are defined by
\begin{equation} \label{poisson-tails}
P(\lambda,n) = \sum_{j=0}^n p(\lambda,j)~, \hspace{.2in}
\overline{P}(\lambda,n) = \sum_{j=n}^\infty p(\lambda,j)~,
\end{equation} 
respectively. Note that $P(\lambda,n-1)+\overline{P}(\lambda,n)=1$. 
\end{definition}

The following quantities will also be of interest:
\begin{equation} \label{BJ}
B = \sum_{i=1}^k \int_0^T \beta_i(s)ds\enspace, \hspace{.2in} J = \sum_{i=1}^k \int_0^T \nu_i(s)ds \enspace.
\end{equation}
It now follows from \eqref{E:LRatio-specific}-\eqref{E:LRatio-i-specific} that
\[
L_T = e^{-J}\prod_{i=1}^k \prod_{n=1}^{N_T(i)}\left(1+\frac{\nu_i(\tau_n(i))}{\beta_i(\tau_n(i))}\right) \enspace.
\]
In the sequel we will also use the integer ceiling function $\lceil \cdot \rceil$ which assigns to a real number $x$ the smallest integer greater than or equal to $x$.

For $\gamma > 0$, consider the test $A_1^* = \{L_T \ge \gamma\}$.  
A lower bound on the probability of detection $\BP_1(L_T \ge \gamma)$, 
and an upper bound on the probability of false alarm $\BP_0(L_T \ge \gamma)$
can now be obtained as follows. 
Recalling Assumptions \ref{A:Process-under-H0} and \ref{A:Process-under-H1}, define
\begin{subequations} \label{bounds}
\begin{align}
C &= 1+\frac{\nu_\sfmin}{\beta_\sfmax} \le \min_{1 \le i \le k} \medspace \min_{0 \le t \le T} \left(1+\frac{\nu_i(t)}{\beta_i(t)}\right)\enspace,\\
D &= 1 + \frac{\nu_\sfmax}{\beta_\sfmin} \ge \max_{1 \le i \le k} \medspace\max_{0 \le t \le T} \left(1+\frac{\nu_i(t)}{\beta_i(t)}\right)\enspace.
\end{align}
\end{subequations}
Note that if one can find $\ell^-$, $\ell^+$ such that
$
\ell^- \le L_T \le \ell^+,
$ 
then for $j \in \{0,1\}$,
\[
\BP_j(\ell^- \ge \gamma) \le \BP_j(L_T \ge \gamma) \le \BP_j(\ell^+ \ge \gamma) \enspace.
\]
Letting
\begin{align*}
\ell^- &= e^{-J} \prod_{i=1}^k C^{N_T(i)}\;, & 
\ell^+ &= e^{-J} \prod_{i=1}^k D^{N_T(i)}\;.
\end{align*}
it can be verified that $\ell^- \le L_T \le \ell^+$. Next, note that
\begin{align*}
\ell^- \ge \gamma \; &\iff \; \sum_{i=1}^k N_T(i) \ge \frac{\log \gamma + J}{\log C} \enspace, \\
\ell^+ \ge \gamma \; &\iff \; \sum_{i=1}^k N_T(i) \ge \frac{\log \gamma + J}{\log D} \enspace.
\end{align*}
Since $N_T(i)$ for $1 \le i \le k$ are independent Poisson random variables 
with parameters $\int_0^T \beta_i(s)ds$ with respect to $\BP_0$, 
it follows that $\sum_{i=1}^k N_T(i)$ is a Poisson random variable 
with parameter $B = \sum_{i=1}^k \int_0^T \beta_i(s)ds$ with 
respect to $\BP_0$. Under the probability measure $\BP_1$, $N_T(i)$ for $1 \le i \le k$ 
are independent Poisson random variables with parameters 
$\int_0^T [\beta_i(s) + \nu_i(s)]ds$.  
It follows that under $\BP_1$, $\sum_{i=1}^k N_T(i)$ is a Poisson random variable with parameter $J+B=\sum_{i=1}^k \int_0^T [\nu_i(s) + \beta_i(s)]ds$. 
Hence,
\begin{subequations} 
\begin{align}
\BP_1(L_T \ge \gamma) &\ge \overline{P}\left( J+B, \left\lceil \frac{\log \gamma + J}{\log C} \right\rceil\right)\enspace,       
 \label{probability-of-detection} \\
\BP_0(L_T \ge \gamma) &\le \overline{P} \left(B, \left\lceil \frac{\log \gamma + J}{\log D} \right\rceil\right)\enspace.       
\label{probability-of-false-alarm}
\end{align}
\end{subequations}


\section{Application to Nuclear Detection}\label{S:Example}

In this section, the framework developed above is applied to the problem of detecting radioactive materials in transit, using a network of spatially distributed sensors. The setting here is simple but representative of a frequently encountered class of scenarios. Our method is not restricted, however, to this setting. Indeed, the results in the previous sections apply whenever the intensity of the suspected source 
and the motion of the source relative to the sensors are deterministic and known.

Radiation sensors always record background radiation (due to cosmic radiation and due to naturally occurring radioactive isotopes in the environment). In the absence of illicit nuclear material (hypothesis $H_0$ is true), the sensors simply measure background.  If radioactive material is present (hypothesis $H_1$ is true), the sensors record the sum of the photons coming from background and the photons coming from the material. These two sources of radiation act independently, and one can treat each sensor as observing a single Poisson process whose intensity is the sum of intensities due to background and material (the source).  The problem we face is to determine, in a fixed amount of time, whether a target passing in front of the sensors is a source of radiation.

The specific assumptions for this problem are as follows: The workspace is the horizontal plane, $\BR^2$. We have $k$ sensors uniformly spaced along the positive $x$-axis at locations $x=0$, $x=\ell$, $x=2\ell,\dots,x=(k-1)\ell$ 
in a configuration as that shown in Fig. \ref{fig:architecture}. 
To span a length of approximately $100$~m, we choose $\ell = 11$~m, $k = 10$,
and for simplicity, we assume that the sensors are identical. Let $t \in [0,T]$ denote time, where $t=0$ corresponds to the instant the count recording is initiated, and $t=T$ is the final time at which a decision regarding the existence of a source is to be made. Let $\beta_i(t)$ be the intensity of background radiation at the location of sensor $i$, $1 \le i \le k$, which does not have to be uniform and in general can be time-dependent. For simplicity, we assume in this example that background intensity is time-invariant, so $\beta_i(t) = \beta_i \in \BN$, where $\beta_i$ is assumed to be varying between locations, from a minimum of $\beta_\mathsf{min} = 2$~\ac{cps} to a maximum of $\beta_\mathsf{max} = 8$~\ac{cps}, with the maximum appearing at the first and last sensor and the minimum occurring at the sensor in the middle (Fig.~\ref{figure:background}). We assume that a target is passing at a distance 
$h = 0.362$~m (the equivalent of 14$\scriptstyle{\frac{1}{4}}$~inches)
from the $x$-axis, namely with a constant coordinate $y=h$, appearing first at some initial location 
$(x_0,h) = (-4,0.362) \in \BR^2$, 
and moving with constant speed $v=17$~m/s (roughly $38$~mph)
in the direction of the positive $x$-axis.   

To illustrate the derivation process, let us for the sake of argument assume that the acceptable probability of false alarm in this scenario is $\alpha =10^{-6}$ (see \eqref{probability-of-false-alarm}). With $r_i(t)$, $1 \le i \le k$, denoting the distance between sensor $i$ and the potential source, the intensity $\nu_i(t)$ at sensor $i$ due to the source is modeled in \cite{NDTW} by\footnote{In fact, for a planar detection scenario the sold angle scales proportionally to $1/r_i$.}
\begin{equation} \label{perceived-intensity}
\nu_i(t) =\frac{\chi a}{r_i(t)^2}
\end{equation}
where $a > 0$ is the activity of the potential source (in 
cps) and $\chi > 0$ (in m$^2$) is the sensors' 
cross-section coefficient.\footnote{For the case of heterogenous sensors, each will have its own $\chi_i$.}
We assume a numerical value for $\chi a$ equal to what has been used in \cite{NDTW}, 
but shielded in 3 cm of lead, dropping the source's perceived intensity by one order of magnitude to
$\chi a = 506.8$~\ac{cps}$\cdot\mathrm{m}^2$.
We also assume that no sensor is ever closer than distance $h$ to the target, 
ensuring that $\nu_i(t)$ is always bounded.

Since the location of the potential source at time $t \in [0,T]$ is $(x_0 + vt,h)$, the distance $r_i(t)$ between the potential source and sensor $i$, $1 \le i \le k$ is given by
$ 
r_i(t) = \sqrt{(x_0+vt-(i-1)\ell)^2 + h^2}.
$ 
Recalling \eqref{perceived-intensity} with 
\[
T = \frac{(k-1)\ell - 2 x_0}{v} = 6.3\enspace \mathrm{s}
\] 
for the decision time, we get  
\begin{align*}
\beta_i(t) &\equiv \beta_i \enspace, & 
\nu_i(t) &= \frac{\chi a}{(x_0 + vt-(i-1)\ell)^2 + h^2}
\end{align*}
for $t \in [0,T]$. 
Since 
\begin{equation*}
\int_0^T \nu_i(s)ds =  \frac{\chi a}{h v}\left[\tan^{-1}\left(\frac{x_0+vT -(i-1)\ell}{h}\right) 
 - \tan^{-1}\left(\frac{x_0-(i-1)\ell}{h}\right)\right],
\end{equation*}
from \eqref{BJ} we obtain
\[
J  = {\textstyle \frac{\chi a}{h v}} \sum_{i=1}^k \left[ \textstyle \tan^{-1} \left(\frac{x_0 + vT-(i-1)\ell}{h}\right) -  
\tan^{-1}\left(\frac{x_0-(i-1)\ell}{h}\right) \right]
\]
and for the ten-sensor array we have $J = 2559.74$~counts. 
\begin{figure}[h!]
\subfigure[background intensity]{ \label{figure:background}
\includegraphics[width=0.45\textwidth]{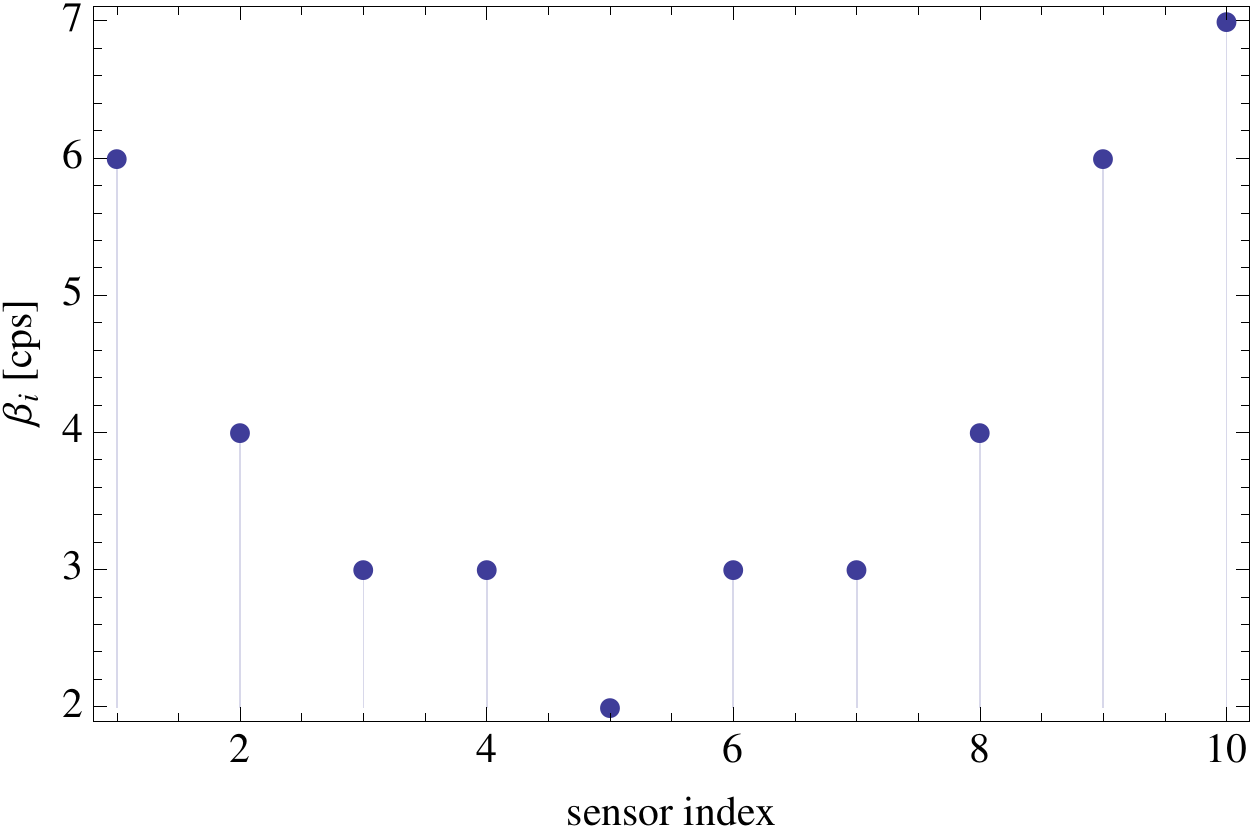}
}
\subfigure[$\int_0^1 \nu_i(s) ds$]{\label{subfig:intni}
\includegraphics[width=0.5\textwidth]{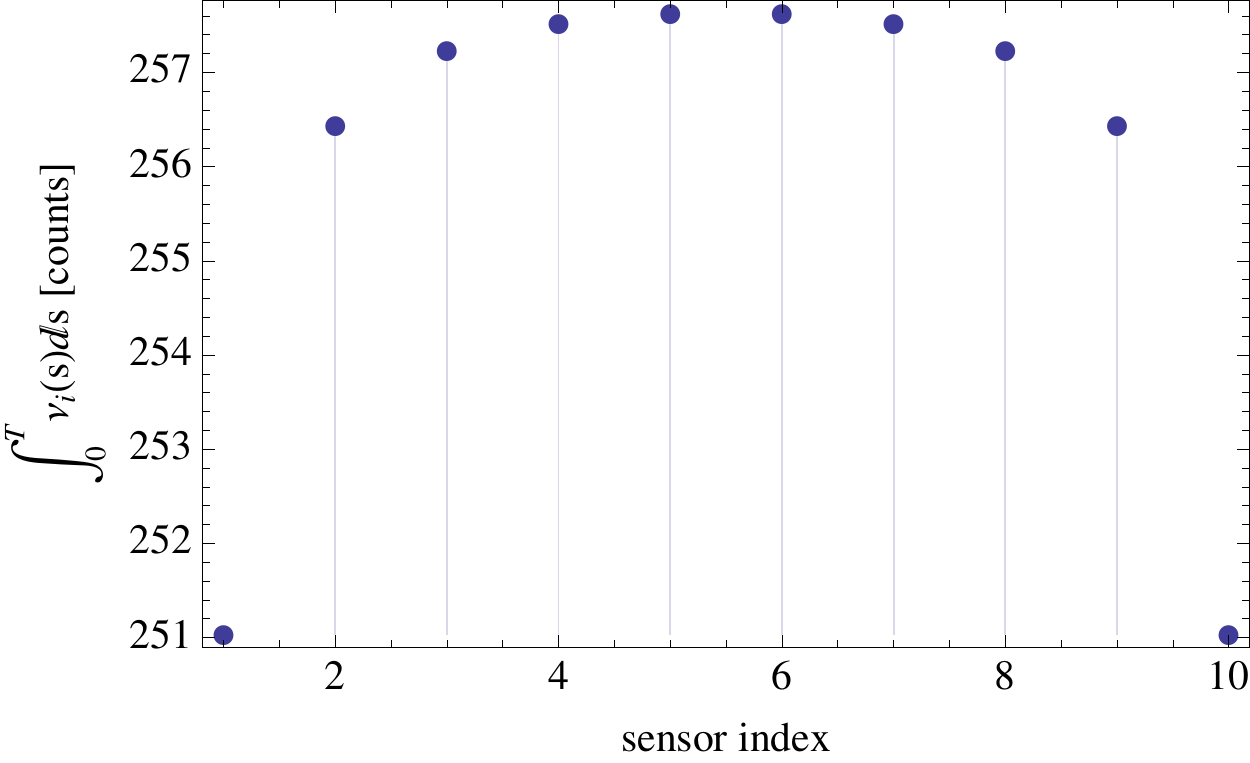}
}
\caption{Distribution of background intensity (\ref{figure:background}) 
and integrated perceived source intensity (\ref{subfig:intni}) 
over the considered ten-sensor array.}
\end{figure}

It is not hard to see that
\begin{align*}
r_\sfmin &= \min_{1 \le i \le k} \medspace \min_{0 \le t \le T} r_i(t) = h\\
 r_\sfmax &= \max_{1 \le i \le k} \medspace \max_{0 \le t \le T} r_i(t) = \sqrt{(x_0 + vT)^2 + h^2}~.
\end{align*}
Thus, for the case of the ten-sensor array, \eqref{bounds} evaluates to 
\begin{align*}
C  &= 1 + \frac{\chi a}{[(x_0 + vT)^2 + h^2] \beta_\sfmax} = 1 + 5.97 \times 10^{-3} \\
D &= 1 + \frac{\chi a}{h^2 \beta_\sfmin} = 1 + 935.24 \enspace.
\end{align*}
With reference to \eqref{BJ} and Fig.~\ref{figure:background}, we have
$B = \frac{4387}{17}$~counts, and with this we can attempt to numerically compute 
a threshold $\gamma$ for the likelihood ratio test using \eqref{probability-of-false-alarm}.
It can be verified that the Poisson tail on the right hand side of 
\eqref{probability-of-false-alarm} falls below $10^{-6}$ when the second
argument of $\overline{P}(\cdot,\cdot)$ increases to $338$ (see Fig.~\ref{figure:pfavn}).
We thus compute the value of $\gamma$ for which 
$\left\lceil \frac{\log \gamma + J}{\log D} \right\rceil \ge 338$ counts, and obtain that with
$\gamma = 0.1718$, the bound on the probability of false alarm $\BP_0$
falls at $8.5 \times 10^{-7}$, which is below the acceptable error rate.  The decision rule therefore is based on the test:
\begin{equation} \label{test}
L_1 \ge 0.1718
\end{equation}
which if true, suggests that the target is indeed a radioactive source.  
\begin{figure}[h!]
\subfigure[Poisson tail]{ \label{figure:pfavn}
\includegraphics[width=0.47\textwidth]{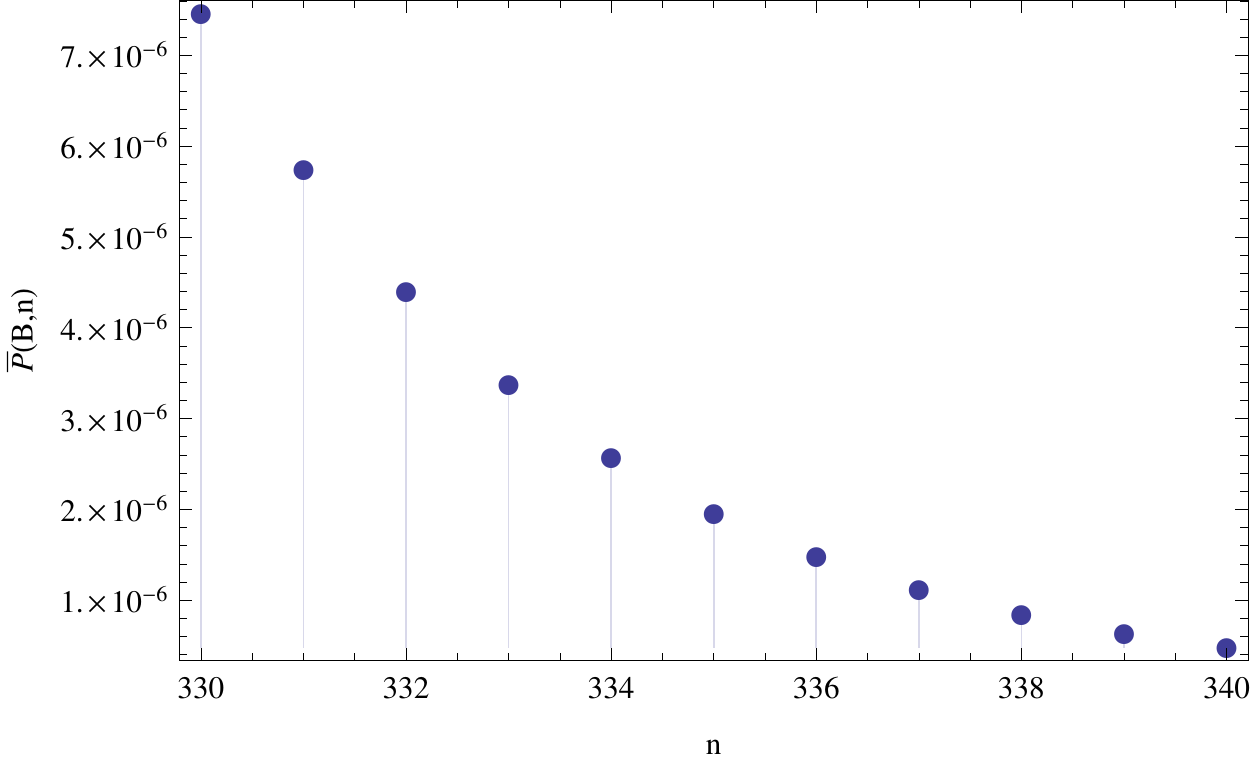}
}
\subfigure[False alarm rate]{ \label{figure:pfa}
\includegraphics[width=0.45\textwidth]{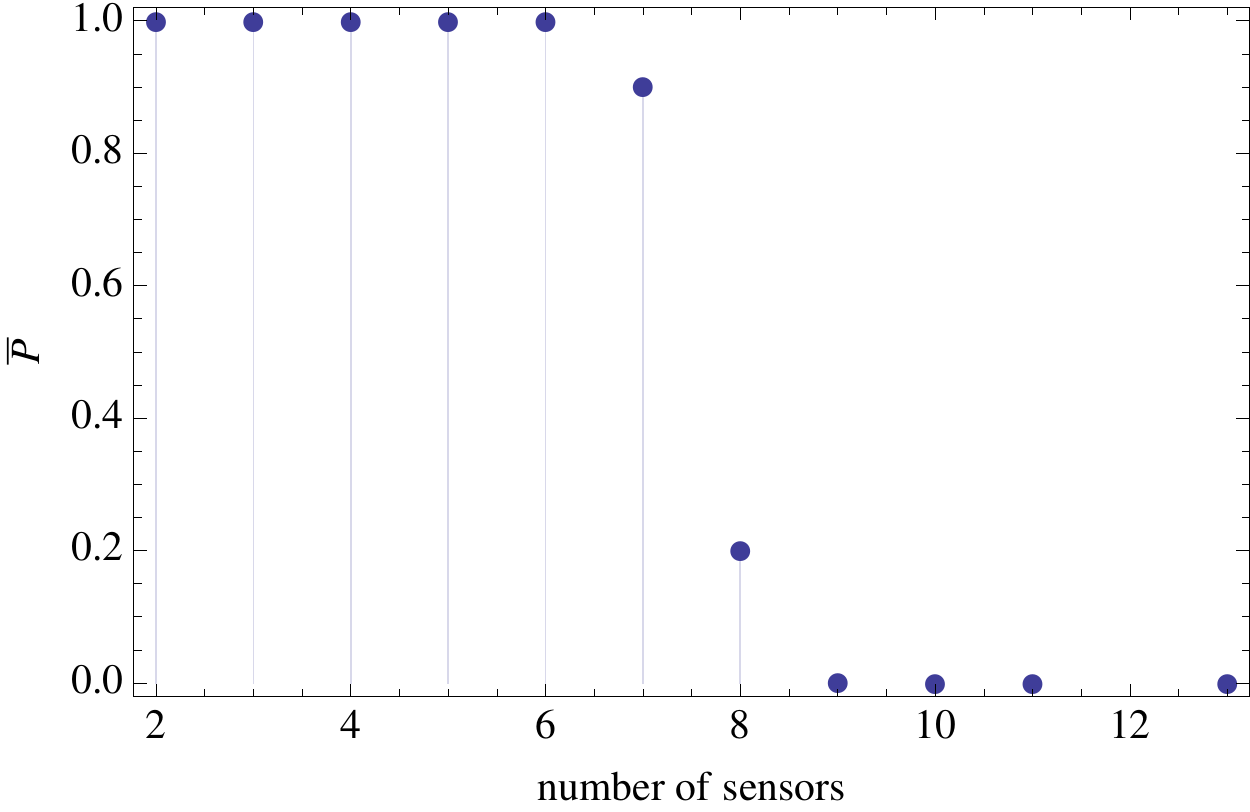}
}
\caption{Bound on the probability of false alarm (\textsc{pfa}). 
The Poisson tail that upper bounds the probability of false alarm decreases monotonically with its
second argument (Fig. \ref{figure:pfavn}), and a value for the latter can be identified for which the 
former falls below a desired value.
Without resetting the threshold constant, we see that the upper bound on the \textsc{pfa} decreases 
with the addition of new sensors. For $k=8$ sensors and $\gamma=0.1718$ a \textsc{pfa} is at most 
20\%, while for two additional sensors one can guarantee a one-in-a-million chance for a false alarm.}
\end{figure}
It should be mentioned that there is conservatism built in the bounds \eqref{bounds}, which renders the probability of detection using \eqref{test} rather impractically 
small for the given false alarm rate.  In addition, it is acknowledged that the illustrated method for obtaining
a threshold makes the solution for $\gamma$ very sensitive to changes in the underlying parameters $h$, $v$, $\chi a$
and $\beta_\mathsf{min}$.
Improving the bounds in \eqref{bounds} is part of ongoing work. Nevertheless, the analysis still gives insight into the effect of different parameters on the probability of detection. To illustrate that point, let us consider the possibility of using more sensors with the same spacing $\ell$ as before.  Without changing the decision rule \eqref{test} (keeping the same threshold), the analysis shows (Fig.~\ref{figure:pfa}) how the upper bound on the probability of false alarm (\textsc{pfa}) estimated in \eqref{probability-of-false-alarm} not only falls monotonically with the addition of new sensors, but that there is
a clear transition between the state where the sensor network decision is unreliable, and that where an alarm should
be taken into account seriously.  Such information can be useful for determining the maximum number of detectors
that can get out of commission before significantly compromising the effectiveness of the system.

\section{Conclusions}\label{S:conclusion}

A network of sensors can be deployed to optimally decide between two hypotheses regarding the statistics of a time-inhomogeneous point process in a way that preserves the accuracy of centralized decision making without incurring the increased communication cost. The sensors collect their measurements over a fixed-time interval, at the end of which a processed summary is communicated to a fusion center. In particular, each sensor transmits a locally computed likelihood ratio to the fusion center, which then compares the product of the sensor-specific likelihood ratios against a threshold to arrive at a decision. The analysis is based on the Neyman-Pearson formulation. A set of conservative performance bounds on the error probabilities is provided and the framework is applied to the problem of detecting a moving radioactive source using an array of sensors.  The work here supports the development of a general decision-making framework that leverages networks of mobile sensor platforms to enhance detection capability in problems that involve time-inhomogeneous point processes.

\section{Appendix}

\begin{proof}[Proof of Lemma \ref{L:L_t-martingale}]
The process $L_t$ given by \eqref{E:LRatio-specific} and \eqref{E:LRatio-i-specific} is a particularization to our problem of the general process $L_t$ in Proposition \ref{P:VI.2.T3}. The latter admits the representation \cite[Equation VI.2.4]{B81}
\begin{equation}\label{E:VI.2.4}
L_t = 1 + \sum_{i=1}^k \int_0^t L_{s-} \medspace(\mu_s(i)-1) dM_s(i)
\end{equation}
where $M_t(i) = N_t(i) - \int_0^t \lambda_s(i)ds$, and for $t>0$, $L_{t-} = \lim_{s \nearrow t} L_s$ is the left limit of $L_t$. The application of (\ref{E:VI.2.4}) to our problem, with $L_t$ given by \eqref{E:LRatio-specific}, \eqref{E:LRatio-i-specific}, yields 
\begin{equation}\label{E:L_t-stoch-int-eq}
L_t = 1 + \sum_{i=1}^k \int_0^t L_{s-} \medspace \frac{\nu_i(s)}{\beta_i(s)} dM_s(i)
\end{equation}
where $M_t(i)=N_t(i)-\int_0^t  \beta_i(s)ds$.  The non-negativity of $L_t$ is evident from \eqref{E:LRatio-specific}, \eqref{E:LRatio-i-specific}. To complete the proof, it thus suffices to show that each of the integrals on the right in \eqref{E:L_t-stoch-int-eq} is a martingale. By \cite[Theorem II.3.T8]{B81}, for $1 \le i \le k$, $\int_0^t L_{s-} (\nu_i(s)/\beta_i(s)) dM_s(i)$ is a $(\BP_0,\filt^N_t)$-martingale whenever 
\begin{equation}
\BE_0\left[\int_0^t L_{s-} \medspace\nu_i(s) ds\right] < \infty 
\end{equation}
for $t \ge 0$.\footnote{Actually, Theorem II.3.T8 in \cite{B81} also requires that $L_{t-}$ be $\filt^N_t$-predictable. This follows from the left-continuity and $\filt^N_t$-adaptedness of $L_{t-}$, by \cite[Theorem I.3.T5]{B81}.} By Assumptions \ref{A:Process-under-H0}, \ref{A:Process-under-H1}, we get that
\begin{equation}
L_t \le \prod_{i=1}^k K^{N_t(i)}\enspace ,
\end{equation}
where $K = 1 + \nu_\sfmax/\beta_{\sfmin}$. Since $K>1$ and the $N_t(i)$, $1 \le i \le k$ are independent with each $N_t(i)$ non-decreasing in $t$, we get
\begin{eqnarray}
\BE_0\left[\int_0^t L_{s-} \medspace\nu_i(s)  ds\right] & \le & \nu_\sfmax  \medspace t \prod_{i=1}^k \BE_0[K^{N_t(i)}]\\
& \le & \nu_\sfmax  \medspace t \prod_{i=1}^k \exp\left[(K-1) \int_0^t \beta_i(s)ds\right] < \infty
\end{eqnarray}
where the last line follows from the fact that under $\BP_0$, for $t \ge 0$, each $N_t(i)$ is a Poisson random variable with parameter $\int_0^t \beta_i(s)ds$.
\end{proof}

\begin{proof}[Proof of Lemma \ref{L:Poisson-rescaling-time}]
Note that $\tilde{\lambda}(u)$ is nonnegative and measurable with $\int_0^u \tilde{\lambda}(v)dv<\infty$ for all $u > 0$. Next, since $X_{T\cdot u}$ is $\filt_{T\cdot u}$-measurable for all $u \ge 0$, it follows that $Y_u$ is $\mathscr{G}_u$-adapted. To complete the proof, we need to show that for $0 \le v \le u$, $Y_u - Y_v$ is independent of $\mathscr{G}_v$ and is a Poisson random variable with parameter $\int_v^u \tilde{\lambda}(\tau)d\tau$. Since $X_{T\cdot u} - X_{T\cdot v}$ is independent of $\filt_{T\cdot v}$, it follows that $Y_u - Y_v$ is independent of $\mathscr{G}_v$. Finally, for $n \in \BZ^+$, we have
\begin{eqnarray}
\BP(Y_u - Y_v = n) &=& \BP(X_{T\cdot u}-X_{T\cdot v}=n) \nonumber \\
&=& \exp\left(-\int_{T\cdot v}^{T\cdot u} \lambda(s)ds\right) \frac{\left(\int_{T\cdot v}^{T\cdot u} \lambda(s)ds\right)^n}{n!} \nonumber\\
&=& \exp\left(-\int_v^u \tilde{\lambda}(\tau)d\tau\right) \frac{\left(\int_v^u \tilde{\lambda}(\tau)d\tau\right)^n}{n!} \enspace . \nonumber
\end{eqnarray}
where the last equality follows by making the change of variables $\tau = s/T$.
\end{proof}

\bibliographystyle{IEEEtran}
\bibliography{DecisionMaking,decision_graph,Extras}

\end{document}